\documentclass[12pt, leqno]{amsart}

\usepackage{texdraw,multicol,rotating}

\def\oh{\overline{h}}
\def\qed{\hfill $\sqcap \hskip-6.5pt \sqcup$}

\usepackage{amsmath}
\usepackage{amssymb}
\usepackage{enumerate}
\usepackage[all]{xy}

\input{txdtools.tex}

\numberwithin{equation}{section}

\theoremstyle{plain}
\newtheorem{thm}{Theorem}[section]
\newtheorem{prop}[thm]{Proposition}
\newtheorem{lem}[thm]{Lemma}
\newtheorem{cor}[thm]{Corollary}

\theoremstyle{definition}
\newtheorem{defi}[thm]{Definition}

\newtheorem{example}[thm]{Example}

\theoremstyle{remark}

\newbox{\tmpa}
\newbox{\tmpb}

\DeclareMathOperator{\wt}{wt}
\newcommand{\nc}{\newcommand}
\nc{\Uq}{U_q} \nc{\Z}{\mathbf{Z}} \nc{\C}{\mathbf{C}}
\nc{\Q}{\mathbf{Q}}
\nc{\op}{\oplus} \nc{\ot}{\otimes} \nc{\pv}{P^{\vee}}
\nc{\ali}{\alpha_i} \nc{\B}{\mathbf{B}} \nc{\F}{\mathbf{F}}
\nc{\bP}{\mathbf{P}} \nc{\V}{\mathbf{V}} \nc{\La}{\Lambda}
\nc{\la}{\lambda} \nc{\nbinom}[2]{\genfrac{}{}{0pt}{1}{{#1}}{{#2}}}
\nc{\qbinom}[2]{\left[\genfrac{}{}{0pt}{1}{{#1}}{{#2}}\right]}
\nc{\path}{\mathcal{P}} \nc{\fit}{\tilde{f}_i}
\nc{\eit}{\tilde{e}_i} \nc{\fjt}{\tilde{f}_j} \nc{\ejt}{\tilde{e}_j}
\nc{\Y}{\mathbf{Y}} \nc{\A}{\mathbf{A}} \nc{\ra}{\rightarrow}
\nc{\vep}{\varepsilon} \nc{\vphi}{\varphi} \nc{\vp}{\varphi}
\nc{\g}{\mathfrak{g}} \nc{\h}{\mathfrak{h}} \nc{\oP}{\overline{P}}
\nc{\pathp}{\mathbf{p}} \nc{\N}{\mathcal{N}} \nc{\R}{\mathcal{R}}

\nc{\tris}{ \bsegment \move(0 0)\lvec(10 0)\lvec(10 10)\lvec(0
0)\ifill f:0.7 \esegment } \nc{\recs}{ \bsegment \move(0 0)\lvec(10
0)\lvec(10 5)\lvec(0 5)\lvec(0 0)\ifill f:0.7 \esegment }
\nc{\hcvec}[5]{%
\getpos(#1 #3)\spx\spy \getpos(#2 #3)\epx\epy \getpos(#4
#5)\xoff\yoff \realadd \spx \xoff \twox \realadd \epx {-\xoff} \thrx
\realadd \spy \yoff \posy \move({\spx} {\spy}) \clvec ({\twox}
{\posy})({\thrx} {\posy})({\epx} {\epy}) \rmove(0 0) }

\nc{\ahead}[2]{%
\cossin (0 0)({#1} {#2})\cosa\sina \bsegment
  \drawdim in \setunitscale 0.065
  \realmult {-0.5} \cosa \hcosa
  \realmult {-0.5} \sina \hsina
  \move({\hcosa} {\hsina}) \ravec({\cosa} {\sina})
\esegment }
\nc{\boxi}{%
{%
\savebox{\tmppic}{\begin{texdraw} \small \drawdim em \textref h:C
v:C \setunitscale 0.55 \htext(0 0){$i$} \move(-1 -1)\lvec(-1
1)\lvec(1 1)\lvec(1 -1)\lvec(-1 -1)
\end{texdraw}}%
\raisebox{-0.19\height}{\usebox{\tmppic}}%
}%
}
\nc{\boxj}{%
{%
\savebox{\tmppic}{\begin{texdraw} \small \drawdim em \textref h:C
v:C \setunitscale 0.55 \htext(0 0.1){$j$} \move(-1 -1)\lvec(-1
1)\lvec(1 1)\lvec(1 -1)\lvec(-1 -1)
\end{texdraw}}%
\raisebox{-0.19\height}{\usebox{\tmppic}}%
}%
}
\nc{\boxipo}{%
{%
\savebox{\tmppic}{\begin{texdraw} \small \drawdim em \textref h:C
v:C \setunitscale 0.55 \htext(0.15 0){$i\!\!+\!\!1$} \move(-1.4
-1)\lvec(-1.4 1)\lvec(1.4 1)\lvec(1.4 -1)\lvec(-1.4 -1)
\end{texdraw}}%
\raisebox{-0.19\height}{\usebox{\tmppic}}%
}%
} \everytexdraw{ \drawdim in \arrowheadsize l:0.065 w:0.03
\arrowheadtype t:F \textref h:C v:C }
\newsavebox{\tmppic}
\newsavebox{\tmpfig}
\newsavebox{\tmpdraw}
\newsavebox{\tmpfiga}
\newsavebox{\tmpfigb}
\newsavebox{\tmpfigc}
\newsavebox{\tmpfigd}
\newsavebox{\tmpfige}
\newsavebox{\tmpfigf}
\newsavebox{\tmpfigg}
\newsavebox{\tmpfigh}
\newsavebox{\tmpfigi}
\newsavebox{\tmpfigj}
\newsavebox{\tmpfigk}
\newsavebox{\tmpfigl}
\newsavebox{\tmpfigm}
\newsavebox{\tmpfign}
\newsavebox{\tmpfigo}
\newsavebox{\tmpfigp}
\newsavebox{\tmpfigq}
\newsavebox{\tmpfigr}
\newsavebox{\tmpfigs}
\newsavebox{\tmpfigt}
\newsavebox{\tmpfigu}
\newsavebox{\tmpfigv}
\newsavebox{\tmpfigw}
\newsavebox{\tmpfigx}
\newsavebox{\tmpfigy}
\newsavebox{\tmpfigz}

\newsavebox{\tmpfigaa}
\newsavebox{\tmpfigab}
\newsavebox{\tmpfigac}
\newsavebox{\tmpfigad}
\newsavebox{\tmpfigae}
\newsavebox{\tmpfigaf}
\newsavebox{\tmpfigag}
\newsavebox{\tmpfigah}
\newsavebox{\tmpfigai}
\newsavebox{\tmpfigaj}
\newsavebox{\tmpfigak}
\newsavebox{\tmpfigal}
\newsavebox{\tmpfigam}
\newsavebox{\tmpfigan}
\newsavebox{\tmpfigao}
\newsavebox{\tmpfigap}
\newsavebox{\tmpfigaq}
\newsavebox{\tmpfigar}
\newsavebox{\tmpfigas}
\newsavebox{\tmpfigat}
\newsavebox{\tmpfigau}
\newsavebox{\tmpfigav}
\newsavebox{\tmpfigaw}
\newsavebox{\tmpfigax}
\newsavebox{\tmpfigay}
\newsavebox{\tmpfigaz}

\newsavebox{\tmpfigba}
\newsavebox{\tmpfigbb}
\newsavebox{\tmpfigbc}
\newsavebox{\tmpfigbd}
\newsavebox{\tmpfigbe}
\newsavebox{\tmpfigbf}
\newsavebox{\tmpfigbg}
\newsavebox{\tmpfigbh}

\nc{\node}{\lcir r:1 }
\nc{\sline}{\bsegment\savepos(10 0)(*ex *ey)
            \move(1 0)\rlvec(8 0)
            \esegment\move(*ex *ey)}
\nc{\dline}{\bsegment\savepos(10 0)(*ex *ey)
            \move(0.93 0.4)\rlvec(8.14 0)\rmove(0 -0.8)\rlvec(-8.14 0)
            \esegment\move(*ex *ey)}
\nc{\uline}{\bsegment\savepos(0 10)(*ex *ey)
            \move(0 1)\rlvec(0 8)
            \esegment\move(*ex *ey)}
\nc{\lpoint}{\savecurrpos(*ex *ey)
             \rmove(2.5 1.5)\rlvec(-1.5 -1.5)\rlvec(1.5 -1.5)
             \move(*ex *ey)}
\nc{\rpoint}{\savecurrpos(*ex *ey)
             \rmove(-2.5 -1.5)\rlvec(1.5 1.5)\rlvec(-1.5 1.5)
             \move(*ex *ey)}
\nc{\bline}{\bsegment\savepos(10 0)(*ex *ey)
            \linewd 0.6 \move(1.1 0)\rlvec(7.8 0)
            \esegment\move(*ex *ey)}

\nc{\araise}[1]{\raisebox{4.5pt}{#1}}
\nc{\braise}[1]{\raisebox{12.1pt}{#1}}
\nc{\craise}[1]{\raisebox{8pt}{#1}}
\nc{\draise}[1]{\raisebox{12pt}{#1}} \nc{\be}{\begin{enumerate}}
\nc{\ee}{\end{enumerate}} \nc{\bnum}{\begin{enumerate}[{\rm(i)}]}
\nc{\cl}{\colon} \nc{\seteq}{\mathbin{:=}} \nc{\re}{\mathrm{re}}
\nc{\im}{\mathrm{im}} \nc{\ran}{\rangle} \nc{\lan}{\langle}
\nc{\on}{\operatorname}
\newcommand{\set}[2]{\left\{#1\,;\,#2\,\right\}}
\nc{\Hom}{\on{Hom}} \nc{\Oint}{\mathcal{O}_{\mathrm{int}}}
\nc{\Wt}{\on{Wt}} \nc{\pP}{\widetilde{P}} \nc{\eq}{\begin{eqnarray}}
\nc{\eneq}{\end{eqnarray}} \nc{\eqn}{\begin{eqnarray*}}
\nc{\eneqn}{\end{eqnarray*}} \nc{\Lemma}{\begin{lem}}
\nc{\enlemma}{\end{lem}}
\newcommand{\isoto}[1][]{\mathop{\xrightarrow[#1]%
{\rule{0pt}{.9ex}%
{\raisebox{-.35ex}[0ex][-.6ex]{$\mspace{1mu}\sim\mspace{2mu}$}}}}}

\nc{\hs}{\hspace*} \nc{\bfi}{\mathbf{i}} \nc{\eps}{\varepsilon}
\nc{\ba}{\begin{array}} \nc{\ea}{\end{array}}
\renewcommand{\Im}{\operatorname{Im}}
\nc{\tf}{\tilde{f}} \nc{\id}{\operatorname{id}} \nc{\bl}{\bigl}
\nc{\br}{\bigr} 
\nc{\Irr}{\on{Irr}}
\nc{\Id}{\on{Id}}
\nc{\out}{\operatorname{out}}
\nc{\sink}{\operatorname{in}}
\nc{\ol}{\overline}

\begin{document}

\title[Geometric Construction of Crystal Bases]
      {Geometric Construction of Crystal Bases for \\ Quantum Generalized Kac-Moody Algebras}
\author[S.-J. Kang, M. Kashiwara, O. Schiffmann]{Seok-Jin Kang$^{1}$,
Masaki Kashiwara$^{2}$, Olivier Schiffmann}

\address{Department of Mathematical Sciences
         and
         Research Institute of Mathematics \\
         Seoul National University \\ San 56-1 Sillim-dong, Gwanak-gu \\ Seoul 151-747, Korea}

         \email{sjkang@math.snu.ac.kr}

\address{Research Institute for Mathematical Sciences \\
         Kyoto University \\ Kitashirakawa, Sakyo-Ku \\ Kyoto 606-8502, Japan}

         \email{masaki@kurims.kyoto-u.ac.jp}

\address{Universit\'e Pierre et Marie Curie \\
         D\'epartement de Math\'ematiques\\ 175 rue du Chevaleret  \\ 75013 Paris,  France}

         \email{olive@math.jussieu.fr}

\thanks{$^{1}$This research was supported by KRF Grant \# 2007-341-C00001.}
\thanks{$^{2}$This research was partially supported by Grant-in-Aid for Scientific Research (B)
18340007, Japan Society for the Promotion of Science.}

\begin{abstract}

We provide a geometric realization of the crystal $B(\infty)$ for
quantum generalized Kac-Moody algebras in terms of the irreducible
components of certain Lagrangian subvarieties in the
representation spaces of a quiver.

\end{abstract}

\maketitle

\vskip 1cm


\section*{Introduction}

There is a well-known and very fruitful interaction between the
structure theory of quantum groups on the one hand and the geometry
of quiver representations on the other. In the late 80's, Ringel
realized the positive part $U_q^{+}(\mathfrak{g})$ of the quantized
enveloping algebra of a Kac-Moody algebra $\mathfrak{g}$ in the Hall
algebra of any quiver whose underlying graph is the Dynkin diagram
of $\mathfrak{g}$ (\cite{Ringel}). This was soon followed by
Lusztig's geometric construction of the {\it canonical basis}
$\mathbf{B}$ for $U_q^{+}(\mathfrak{g})$ in terms of simple perverse
sheaves on the moduli spaces $\mathcal{M}_{\alpha}$ of
representations of quivers (\cite{Lus}). The combinatorial structure
of this canonical basis is encoded in a colored graph $B(\infty)$,
the {\it crystal graph} of $U_q^{-}(\mathfrak{g})$, whose vertices
are the elements of $\mathbf{B}$ (\cite{Kas91}).

By studying the cotangent geometry of $\mathcal{M}_{\alpha}$,
Kashiwara and Saito later gave a geometric construction of the
crystal graph $B(\infty)$ (\cite{KS97}). More precisely, they
considered a certain Lagrangian subvariety $\mathcal{N}_{\alpha}
\subset T^* \mathcal{M}_{\alpha}$ (first introduced in \cite{Lus91})
and built a graph $\mathcal{B}$ whose vertices are the irreducible
components of $\bigsqcup_{\alpha} \mathcal{N}_{\alpha}$ and whose
arrows correspond to various generic fibrations between irreducible
components of $\N_{\alpha}$ for different values of $\alpha$. Using
a combinatorial characterization of $B(\infty)$ in terms of tensor
products with {\it elementary crystals}, the authors of \cite{KS97}
identified $\mathcal{B}$ with $B(\infty)$. This work was later
generalized by Saito who realized the crystals of all highest weight
integrable representations using Lagrangian subvarieties in
Nakajima's quiver varieties (see \cite{Sai}, \cite{Nak97}).


In another direction, Borcherds was led in his study of the
Moonshine and the Monster group to consider a new class of
infinite-dimensional Lie algebras, now called the {\it generalized
Kac-Moody algebras} (\cite{Bo88}). Although similar to Kac-Moody
algebras in several respects, generalized Kac-Moody algebras allow
for \textit{imaginary} simple roots and play an important role in
several different areas of mathematics (see, for example,
\cite{Bo88, Bo92, FRS97, KKwon00, Na95}). The notion of the
quantized enveloping algebra of a generalized Kac-Moody algebra
${\mathfrak g}$ was defined in \cite{Kang95}. Several important
structural properties of quantum groups were shown to exist also in
this generalized setting. For instance, the integrable highest
weight modules over ${\mathfrak g}$ can be deformed to those over
$U_q(\g)$ in such a way that the dimensions of weight spaces are
invariant under the deformation.

In \cite{JKK}, the crystal basis theory was developed for quantum
generalized Kac-Moody algebras and in \cite{JKKS}, the notion of
{\it abstract crystals} was introduced. Moreover, in \cite{JKKS},
the authors proved a crystal embedding theorem and gave a
characterization of the highest weight crystals $B(\infty)$ and
$B(\lambda)$ for quantum generalized Kac-Moody algebras.

This paper is part of the project to extend the approach based on
the geometry of quiver representations to the case of generalized
Kac-Moody algebras and their quantized enveloping algebras. A
(partly conjectural) geometric construction of the canonical basis
$\mathbf{B}$ of $U_q^{+}(\mathfrak{g})$ for an (even) generalized
Kac-Moody algebra $\mathfrak{g}$ was given in \cite{KaSc} (see also
\cite{Lus2} and \cite{LiYi}). At the level of quivers, moving from
Kac-Moody algebras to generalized Kac-Moody algebras means that one
must now consider the quivers with edge loops and
\textit{semisimple} rather than simple perverse sheaves. In the
present work, we provide an analogue of the construction in
\cite{KS97} of the crystal $B(\infty)$. Namely, we consider a
certain Lagrangian subvariety $\mathcal{N}_{\alpha} \subset T^*
\mathcal{M}_{\alpha}$, and build a crystal graph out of its
irreducible components and generic fibrations among them. The
noticeable difference with \cite{KS97} is that the fibrations in
question correspond, at the level of the representation theory of
quivers, to extensions by \textit{nonrigid} simple objects; i.e.,
the objects with non-vanishing self $\text{Ext}^1$-- typically a
simple object sitting at a vertex with edge loops. We get around
this difficulty by restricting our Lagrangian variety to the
cotangent bundle of a certain open subset of $\mathcal{M}_{\alpha}$
by imposing that certain arrows are {\it regular semisimple} (see
Section~2).

The paper is organized as follows. In Section 1, we recall various
definitions and results pertaining to the crystal basis theory for
quantum generalized Kac-Moody algebras as developed in \cite{JKKS}.
In particular, we recall the characterization of the crystal
$B(\infty)$ in terms of strict crystal embeddings $B(\infty) \to
B(\infty) \otimes B_i$ for $i\in I$ (see
Theorem~\ref{thm:B(infty)}). Sections 2 and Section 3 are devoted to
the definition and study of the Lagrangian variety
$\mathcal{N}_{\alpha}$. The crystal structure on the set
$\mathcal{B}$ of irreducible components of $\bigsqcup_{\alpha}
\N_{\alpha}$ is described in Section~3 (see
Theorem~\ref{thm:crystal}). Finally, in Section~4, we prove the
crystal isomorphism $\mathcal{B} \cong B(\infty)$ by constructing
strict crystal embeddings $\mathcal{B} \to \mathcal{B} \otimes B_i$
for all $i\in I$ (see Theorem~\ref{thm:main}).

\vskip 1cm


\section{Abstract Crystals}

\vskip 3mm

Let $I$ be a finite or countably infinite index set and let
$A=(a_{ij})_{i, j\in I}$ be a {\it symmetric even integral
Borcherds-Cartan matrix}. That is, $A$ satisfies: \ (i) $a_{ii}
\in \{2, 0, -2, -4, \ldots\}$ for all $i\in I$, \ (ii) $a_{ij} =
a_{ji} \in \Z_{\le 0}$ for $i \neq j$. We say that an index $i\in
I$ is {\it real} if $a_{ii}=2$ and {\it imaginary} if $a_{ii}\le
0$. We denote by $I^{re}=\set{i\in I}{a_{ii}=2}$ and $I^{im}=\set{
i \in I}{a_{ii} \le 0 }$ the set of real indices and the set of
imaginary indices, respectively.

A {\it Borcherds-Cartan datum} $(A, P, \Pi, \Pi^{\vee})$ consists of

\begin{itemize}
\item[(i)] a Borcherds-Cartan matrix $A=(a_{ij})_{i, j \in I}$,

\item[(ii)] a free abelian group $P$, the {\em weight lattice},

\item[(iii)] $\Pi=\set{\alpha_i\in P}{i\in I}$, the set of {\em
simple roots},

\item[(iv)] $\Pi^{\vee} = \set{ h_i}{ i\in I }\subset
P^\vee\seteq\Hom(P,\Z)$, the set of {\em simple coroots}
\end{itemize}
satisfying the following properties:
\begin{itemize}
\item[(a)] $\langle h_i,\alpha_j \rangle = a_{ij}$ for all $i, j \in
I$,

\item[(b)] $\Pi$ is linearly independent,

\item[(c)] for any $i\in I$, there exists $\Lambda_i\in P$ such that
$\langle h_j,\Lambda_i \rangle=\delta_{ij}$ for all $j\in I$.
\end{itemize}

\noindent
We use the notation $Q=\bigoplus_{i\in I} \Z
\alpha_i$ and $Q_{+}=\sum_{i\in I} \Z_{\ge 0} \alpha_i$.

Let $q$ be an indeterminate. For an integer $n \in \Z$, define
\begin{equation*}
[n] = \dfrac{q^n - q^{-n}} {q - q^{-1}}, \qquad [n]! =
\prod_{k=1}^n [k], \qquad \left[\begin{matrix} m \\
n \end{matrix}\right] = \dfrac{[m]!}{[n]! [m-n]!}.
\end{equation*}
For a Borcherds-Cartan datum $(A, P, \Pi, \Pi^{\vee})$, the {\it
quantum generalized Kac-Moody algebra} $U_q(\g)$ is defined to be
the associated algebra over $\Q(q)$ with 1 generated by the elements
$e_i$, $f_i$ $(i\in I)$, $q^h$ $(h\in P^{\vee})$ subject to the
defining relations: {\allowdisplaybreaks
\begin{equation} 
\begin{aligned}
& q^0 =1, \quad q^h q^{h'} = q^{h+h'}\quad\text{for} \ \ h, h' \in
P^{\vee}, \\
& q^h e_i q^{-h} = q^{\alpha_i(h)} e_i, \quad q^h f_i q^{-h} =
q^{-\alpha_i(h)} f_i \quad\text{for $h\in P^{\vee}$, $i\in I$,} \\
& e_i f_j - f_j e_i = \delta_{ij} \dfrac{K_i - K_i^{-1}} {q -
q^{-1}} \quad\text{for $i, j \in I$, where $ K_i =
q^{h_i}$,} \\
& \sum_{k=0}^{1-a_{ij}} (-1)^k \left[\begin{matrix} 1-a_{ij} \\ k
\end{matrix} \right] e_i^{1-a_{ij}-k} e_j e_i^k =0 \quad\text{if
$i\in I^{re}$ and $i\neq j$,}\\
& \sum_{k=0}^{1-a_{ij}} (-1)^k \left[\begin{matrix} 1-a_{ij} \\ k
\end{matrix} \right] f_i^{1-a_{ij}-k} f_j f_i^k =0 \quad
\text{if $i\in I^{re}$ and $i\neq j$,} \\
& e_i e_j - e_j e_i = f_i f_j - f_j f_i =0 \quad\text{if
$a_{ij}=0$.}
\end{aligned}
\end{equation}
} \noindent We denote by $U_q^+(\g)$ (resp.\ $U_q^-(\g)$) the
subalgebra of $U_q(\g)$ generated by the $e_i$'s (resp.\ the
$f_i$'s).

\vskip 3mm

We recall the notion of {\it abstract crystals} for quantum
generalized Kac-Moody algebras introduced in \cite{JKKS}.

\begin{defi} \label{defi:abstract crystal} An {\em abstract
$U_q(\g)$-crystal} or simply a {\em crystal\/} is a set $B$ together
with the maps $\wt \cl B \rightarrow P$, $\eit, \fit \cl B \rightarrow B
\sqcup \{0\}$ and $\vep_i, \vphi_i \cl  B \rightarrow \Z \sqcup
\{-\infty\}$ $(i\in I)$ satisfying the following conditions:

\begin{itemize}
\item[(i)] $\wt(\eit b) = \wt b + \alpha_i$ if $i\in I$ and
$\eit b \neq 0$,

\item[(ii)] $\wt(\fit b) = \wt b - \alpha_i$ if $i\in I$ and
$\fit b \neq 0$,

\item[(iii)] for any $i \in I$ and $b\in B$, $\vphi_i(b) = \vep_i(b) +
\langle h_i, \wt b \rangle$,

\item[(iv)] for any $i\in I$ and $b,b'\in B$,
$\fit b = b'$ if and only if $b = \eit b'$,\label{cond7}

\item[(v)] for any $i \in I$ and $b\in B$
such that $\eit b \neq 0$, we have \be[{\rm(a)}]
\item
$\vep_i(\eit b) = \vep_i(b) - 1$, $\vphi_i(\eit b) = \vphi_i(b) + 1$
if $i\in I^\re$,
\item
$\vep_i(\eit b) = \vep_i(b)$, \ $\vphi_i(\eit b) = \vphi_i(b) +
a_{ii}$ if $i\in I^\im$, \ee

\item[(vi)] for any $i \in I$ and $b\in B$ such that $\fit b \neq 0$,
we have \be[{\rm(a)}]
\item
$\vep_i(\fit b) = \vep_i(b) + 1$, $\vphi_i(\fit b) = \vphi_i(b) - 1$
if $i\in I^\re$,
\item
$\vep_i(\fit b) = \vep_i(b)$, \ $\vphi_i(\fit b) = \vphi_i(b) -
a_{ii}$ if $i\in I^\im$, \ee

\item[(vii)] for any $i \in I$ and $b\in B$ such that $\vphi_i(b) = -\infty$, we
have $\eit b = \fit b = 0$.

\end{itemize}
\end{defi}

\begin{defi}\label{def:mor}
Let $B_1$ and $B_2$ be crystals. A map $\psi\cl B_1\rightarrow B_2$
is a {\it crystal morphism} if it satisfies the following
properties:

\bnum
\item for $b\in B_1$, we have
\begin{equation*}
\text{$\wt(\psi(b))=\wt(b)$, $\vep_i(\psi(b))=\vep_i(b)$,
$\vp_i(\psi(b))=\vp_i(b)$ for all $i\in I$,}
\end{equation*}
\item for $b\in B_1$ and $i\in I$ with $\fit b\in B_1$, we have
$\psi(\fit b)=\fit\psi(b)$. \label{cond:crysmor2} \ee
\end{defi}

\begin{defi}
Let $\psi\cl B_1\rightarrow B_2$ be a crystal morphism.

\be[{\quad \rm(a)}]
\item $\psi$ is called a {\it strict morphism} if
\begin{equation*}
\psi(\eit b)=\eit\psi(b),\,\, \psi( \fit
b)=\fit\psi(b)\quad\text{for all $i\in I$ and $b\in B_1$.}
\end{equation*}
Here, we understand $\psi(0)=0$.
\item $\psi$ is called an {\em embedding} if the underlying
map $\psi\cl B_1\rightarrow B_2$ is injective.
\ee
\end{defi}

We will often use the notation $\wt_{i}(b) = \langle h_i, \wt(b)
\rangle$ $(i \in I, b\in B)$.

\begin{example}
Fix $i\in I$. For any $u \in U_q^{-}(\g)$, there exist unique $v, w
\in U_q^{-}(\g)$ such that
\begin{equation*}
e_i u - u e_i = \dfrac{K_i v - K_i^{-1} w}{q_i - q_i^{-1}}.
\end{equation*}
We define the endomorphism $ e_i' \cl  U_q^{-}(\g) \ra U_q^{-}(\g)$
by $e_i'(u)=w$. Then every $u\in U_q^{-}(\g)$ has a unique {\it
$i$-string decomposition}
\begin{equation*}
u=\sum_{k\ge 0} f_i^{(k)} u_k, \quad \text{where} \ \ e_i' u_k =0 \
\ \text{for all} \ \ k\ge 0,
\end{equation*}
where
$$f_i^{(k)}\seteq\begin{cases}
f_i^{k}/[k]!&\text{if $i$ is real,}\\
f_i^{k}&\text{if $i$ is imaginary.}
\end{cases}$$
The {\it Kashiwara operators} $\eit$, $\fit$ $(i\in I)$ are then
defined by
\begin{equation*}
\eit u = \sum_{k\ge 1} f_i^{(k-1)} u_k, \qquad \fit u = \sum_{k\ge
0} f_i^{(k+1)} u_k.
\end{equation*}

Let $\A_0 = \set{f/g \in \Q(q) }{ f, g \in \Q[q], g(0) \neq 0 }$ and
let $L(\infty)$ be the $\A_0$-submodule of $U_q^{-}(\g)$ generated
by
$$\set{\tilde f_{i_1} \cdots \tilde f_{i_r} \mathbf{1} }{ r \ge 0,
i_k \in I },$$ where $\mathbf{1}$ is the multiplicative identity in
$U_q^{-}(\g)$. Then the set
\begin{equation*}
B(\infty) = \set{\tilde f_{i_1} \cdots \tilde f_{i_r} \mathbf{1} + q
L(\infty) }{ r\ge 0, i_k \in I } \setminus \{0\}
\subset L(\infty)/qL(\infty)
\end{equation*}
becomes a crystal with the maps $\wt$, $\eit, \fit$, $\vep_i,
\vp_i$ ($i\in I$) defined by
\begin{equation*}
\begin{aligned}
\wt(b) & =- (\alpha_{i_1} + \cdots + \alpha_{i_r}) \quad \text{for}
\ \ b=\tilde f_{i_1} \cdots \tilde f_{i_r}
\mathbf{1} + q L(\infty), \\
\vep_i (b) & =
\begin{cases} \max \set{k\ge 0 }{ \eit^k b \neq 0 }&\text{for $i\in I^\re$,}
\\
0&\text{for $i\in I^\im$,}\end{cases}\\
\vphi_i (b) & = \vep_i(b) + \wt_i(b) \quad (i\in I).
\end{aligned}
\end{equation*}
\end{example}

\begin{example}
For each $i\in I$, let $B_i = \set{b_i(-n) }{ n\ge 0}$. Then $B_i$
is a crystal with the maps defined by
\begin{equation*}
\begin{aligned}
& \wt b_i(-n) = -n \alpha_i, \\
& \eit b_i(-n) = b_i(-n+1), \quad \fit b_i(-n) = b_i(-n-1), \\
& \tilde e_j b_i(-n) = \tilde f_j b_i(-n) = 0 \quad \text{if} \ \
j \neq i,\\
& \vep_i (b_i(-n)) = n, \quad \vphi_i (b_i(-n))=-n \quad \text{if}
\ \ i\in I^\re, \\
& \vep_i (b_i(-n)) = 0, \quad \vphi_i
(b_i(-n))=\wt_i(b_i(-n))=-na_{ii} \quad \text{if}
\ \ i\in I^\im, \\
& \vep_j (b_i(-n)) = \vphi_j (b_i(-n)) = -\infty \quad \text{if} \ j
\neq i.
\end{aligned}
\end{equation*}
Here, we understand $b_i(-n)=0$ for $n<0$. The crystal $B_i$ is
called an {\it elementary crystal}.

\end{example}

For a pair of crystals $B_1$ and $B_2$, their {\it tensor product}
is defined to be the set
$$B_1 \ot B_2 =  \set{b_1\otimes
b_2}{b_1\in B_1, b_2\in B_2},$$ where the crystal structure is
defined as follows: \ The maps $\wt, \vep_i,\vp_i$ are given by
\eqn
\wt(b\otimes b')&=&\wt(b)+\wt(b'),\\
\vep_i(b\otimes b')&=&\max(\vep_i(b),
\vep_i(b')-\wt_i(b)),\\
\vp_i(b\otimes b')&=&\max(\vp_i(b)+\wt_i(b'),\vp_i(b')). \eneqn For
$i\in I$, we define \eqn \fit(b\otimes b')&=&
\begin{cases}\fit b\otimes b'
&\text{if $\vp_i(b)>\vep_i(b')$,}\\
b\otimes \fit b' &\text{if $\vp_i(b)\le \vep_i(b')$,}
\end{cases}
\eneqn For $i\in I^\re$, we define \eqn \eit(b\otimes b')&=&
\begin{cases}\eit b\otimes b'\ &\text{if
$\vp_i(b)\ge \vep_i(b')$,}\\
b\otimes \eit b' &\text{if $\vp_i(b)< \vep_i(b')$,}
\end{cases}
\eneqn and, for $i\in I^\im$, we define \eqn \eit(b\otimes b')&=&
\begin{cases}\eit b\otimes b'\
&\text{if $\vp_i(b)>\vep_i(b')-a_{ii}$,}\\
0&\text{if $\vep_i(b')<\vp_i(b)\le\vep_i(b') -a_{ii}$,}\\
b\otimes \eit b' &\text{if $\vp_i(b)\le\vep_i(b')$.}
\end{cases}
\eneqn

We recall the {\it crystal embedding theorem} proved in
\cite{JKKS}.

\begin{thm} \cite{JKKS}
For each $i\in I$, there exists a unique strict crystal embedding
$$\Psi_i\cl B(\infty)\to B(\infty)\otimes B_i$$
which sends $\mathbf{1}$ to $\mathbf{1} \ot b_i (0)$.
\end{thm}

As an application of the crystal embedding theorem, we obtain a
characterization of the crystal $B(\infty)$.

\begin{thm} \label{thm:B(infty)} \cite{JKKS}
Let $B$ be a crystal. Suppose that $B$ satisfies the following
conditions{\rm :} \bnum
\item  $\wt(B)\subset -Q_{+}$,

\item there exists an element $b_0\in B$ such that
$\wt(b_0)=0$,


\item for any $b\in B$ such that $b\neq b_0$, there exists
some $i\in I$ such that $\eit b\neq 0$,

\item for each $i \in I$, there exists a strict crystal embedding
$\Psi_i\cl B\rightarrow B\otimes B_i$. \ee

Then there is a crystal isomorphism
$$B\isoto B(\infty),$$
which sends $b_0$ to $\mathbf{1}$.
\end{thm}

\vskip 1cm


\section{Quiver variety}

\vskip 3mm
Let $(I, H)$ be a quiver with an orientation $\Omega$.  Namely, we have maps
$$\out, \sink\cl H\to I$$ and an involution $-$ of $H$ such that
$\out(\ol{h})=\sink(h)$ for any $h\in H$.
We assume that $-$ does not have a fixed point, and
$H=\Omega \sqcup \overline{\Omega}$. If $i=\out(h)$ and $j=\sink(h)$, then we
say that $h$ is an arrow from $i$ to $j$ and write $h\cl i\to j$.
If $h\in H$ satisfies $\out(h)=\sink(h)$, then we say that $h$ is a loop.
We denote
by $H^{loop}$ (resp.\ $\Omega^{loop}$) the set of all loops in $H$
(resp.\ in $\Omega$). Let $c_{ij}$ denote the number of arrows in $H$
from $i$ to $j$, and define
\begin{equation*}
a_{ij}=\begin{cases} 2- c_{ii} = 2 - (\text{the number of loops at
$i$ in $H$})
\ \ & \text{if} \ \ i=j, \\
 - c_{ij} = -(\text{the number of arrows in $H$ from $i$ to $j$}) \
\ & \text{if} \ \ i \neq j.
\end{cases}
\end{equation*}
Then $A=(a_{ij})_{i,j \in I}$ becomes a symmetric even integral
Borcherds-Cartan matrix.

For $\alpha \in Q_{+}$, let $V_{\alpha}=\bigoplus_{i\in I} V_{i}$ be an
$I$-graded vector space with
$$\underline{\dim} \, V_{\alpha}\seteq \sum_{i\in I}
(\dim V_i) \alpha_i = \alpha,$$ let $GL_{\alpha}=\prod_i GL(V_{i})$,
and set $$X_{\alpha} = \bigoplus_{h\in
H}\Hom(V_{\out(h)},V_{\sink(h)}).$$

We introduce a symmetric bilinear form $( \ , \ )$ on $Q$
by the formula
$$\big( \sum d_k \alpha_k, \sum e_j \alpha_j \big) =\sum_k d_ke_k.$$
Thus $\dim X_{\alpha}=\left( (2 \Id-A) \cdot \alpha, \alpha \right)$.

The symplectic form $\omega$ on $X_{\alpha}$ is given by 
$$\omega(B,B')=\sum_h \varepsilon(h) Tr(B_{\oh}B'_h),$$
where
$$\varepsilon(h)=\begin{cases} 1 & \text{if} \; h\in \Omega, \\
-1 & \text{if\;}  h\in \overline{\Omega}.
\end{cases}$$

Consider the {\it moment map}
$\mu=(\mu_i \cl X_{\alpha} \rightarrow \mathrm{End}\; V_i)_{i\in I}$ given
by
$$\mu_i(B) = \sum_{\substack{h\in H \\ \out(h)=i}} \varepsilon(h)
B_{\overline{h}} B_{h}.$$

Let $X_{\alpha}^{\circ}$ denote the set of $B$'s
such that $B_{h}$ is regular semisimple for all $h \in
\overline{\Omega}^{loop}$.
Then $X_{\alpha}^{\circ}$ is a Zariski open subset of
$X_\alpha$.
Let $\N_{\alpha}$ be the variety consisting of all $B=(B_{h})_{h\in
H} \in X_{\alpha}$ satisfying the following three conditions:
\begin{itemize}
\item[(i)] there exists an $I$-graded complete flag ${F}=(F_0 \subset
F_1 \subset F_2 \subset \cdots)$ such that
$$B_{h}(F_k) \subset
F_{k} \ \ \text{for all} \ h\in \overline{\Omega}^{loop}, \ \
B_{h}(F_{k}) \subset F_{k-1} \ \ \text{for all} \ h\in H \setminus
\overline{\Omega}^{loop},$$

\item[(ii)] $\mu_i(B)=0$ for all $i\in I$,

\item[(iii)] 
$B\in X_{\alpha}^{\circ}$.
\end{itemize}
Then $\N_{\alpha}$ is a Zariski closed subvariety of
$X_{\alpha}^{\circ}$.

\vspace{.1in}

We first prove:

\begin{lem} \label{lem:isotropic}
The variety $\N_{\alpha}$ is isotropic.
\end{lem}

\begin{proof}
We first recall the following general fact. Let $X$ be a smooth
algebraic variety, $Y$ a projective variety, and $Z$ a smooth closed
algebraic subvariety of $X \times Y$. Consider the Lagrangian
variety $\Lambda = T_{Z}^{*}(X \times Y)$ and the projection map $q\cl
\Lambda \cap (T^{*}X \times T_{Y}^{*}Y) \rightarrow T^{*}X$. Then it
is known that the image of $q$ is isotropic (e.g.\ see
\cite[Proposition 8.3]{KS}).

We apply this fact to the case where $X=\mathcal{M}_{\alpha}
\seteq\prod_{h\in\Omega}\Hom(V_{\out(h)},V_{\sink(h)})$, the
moduli space of representations of the quiver $(I,\Omega)$,
$Y=\mathfrak{B}$, the
variety of $I$-graded complete flags, and
$$Z=\set{((B_{h})_{h\in \Omega}, {F})\in X\times Y}{B_{h} F_{k} \subset
F_{k -1} \ \ \text{for all} \ k \ge 0 }.$$ Then
\begin{equation*}
\begin{aligned}
T^{*}\mathfrak{B} = & \{(F, K)\, ; \,F\in \mathfrak{B}, \ \text{$K$
is
an $I$-graded endomorphism of $V$} \\
&\hs{20ex} \text{such that $K(F_k) \subset F_{k-1}$
for all $k\ge 0$} \}, \\
T^{*} \mathcal{M}_{\alpha}  = & X_{\alpha}, \\
\Lambda  = & \{(B,F,K)\;;\; K = \sum \eps(h)B_hB_{\oh},\\
&\hs{15ex}B_{h} F_{k} \subset F_{k-1},\,B_{\oh} F_{k} \subset
F_{k} \ \ \text{for all} \ h\in \Omega, k \ge 0\}
\end{aligned}
\end{equation*}
and \begin{equation*} \begin{aligned} \mathrm{Im} \, q =  \{
B=(B_{h})_{h\in H} \;;\; &\text{there exists ${F} \in
\mathfrak{B}$ such that $B_{h} (F_k) \subset F_{k-1}$, $B_{\oh}
(F_{k}) \subset F_{k}$} \\
&\; \text{for all $h\in \Omega, k \ge 0$, $\mu_i(B)=0$ for all $i\in
I$}\}. \end{aligned}
\end{equation*}
 Since our variety $\N_{\alpha}$
is contained in $\mathrm{Im} \, q$, which is isotropic, we are done.
\end{proof}

Later we will show that $\N_{\alpha}$ is Lagrangian for all
$\alpha$. For the moment, we consider a simple case~:

\vspace{.1in}

\begin{lem}\label{L:la} For any $l \geq 1$ and $i \in I$,
the variety $\N_{l\alpha_i}$ is Lagrangian and irreducible.
\end{lem}
\begin{proof} We have to prove that $\N_{l\alpha_i}$ is irreducible
and $\dim(\N_{l\alpha_i})=\frac{1}{2}
\dim(X_{\alpha})$.

If $i$ is a real vertex, then
$\N_{l\alpha_i}=X_{l\alpha_i}=\{pt\}$.

Assume now that $i$ is an imaginary vertex and that $c_{ii}=2$. Then
$\Omega_i^{loop}=\{h\}$ and the moment map condition (ii) implies
$[B_{\oh},B_{h}]=0$. Note that
$|\Omega_i^{loop}|=\frac{1}{2}c_{ii}$. Since $B_{\oh}$ is regular
semisimple and $B_{h}$ is nilpotent, $B_h=0$. Therefore
$\N_{l\alpha_i} \simeq \mathfrak{gl}(l)^{reg}$, the set of regular
semisimple elements of $\mathfrak{gl}(l)$. In particular,
$\N_{l\alpha_i}$ is irreducible and of dimension $l^2=\frac{1}{2}
\dim X_{l\alpha_i}$.


Finally, consider the case where $c_{ii}>2$, let $\mathfrak{B}$ be
the flag variety of $\mathfrak{gl}(l)$ and let
$$Y=\set{((B_{h}, B_{\oh})_{h\in \Omega}, \mathfrak{b})}
{ B_{\oh} \in \mathfrak{b}^{reg}, B_{h} \in \mathfrak{n},
\sum_{h\in\Omega} [B_{\oh},B_h]=0},$$
where $\mathfrak{b}$ is a Borel subalgebra containing all $B_{h}$,
$B_{\oh}$'s, $\mathfrak{b}^{reg}$ is the set of regular semisimple
elements in $\mathfrak{b}$ and
$\mathfrak{n}=[\mathfrak{b},\mathfrak{b}]$. Set
$$Y^{\circ} = \set{((B_{h}, B_{\oh})_{h\in \Omega}, \mathfrak{b}) \in Y }{
\text{$B_{h}$'s are regular nilpotent for all $h \in \Omega$} }.$$
Then $Y^{\circ}$ is an open dense subset of $Y$. Let $\pi\cl Y \to
\mathfrak{B}$ be the natural projection. Because each $B_{\oh}$ is
regular simple, $\mathrm{ad}(B_{\oh})\cl \mathfrak{n} \to \mathfrak{n}$ is
invertible. It follows that $\pi^{-1} (\mathfrak{b})$ is a vector
bundle over $(\mathfrak{b}^{reg})^{c_{ii}/2}$ of rank
$(c_{ii}/2-1)\dim \mathfrak{n}$. Hence
$\pi^{-1}(\mathfrak{b})$ is irreducible and of dimension
$\frac{1}{2}c_{ii}\dim \mathfrak{b} + (\frac{1}{2}c_{ii}-1)\dim
\mathfrak{n}=\frac{1}{2}(l^2c_{ii}-l(l-1))$. Therefore, $Y^{\circ}$
is irreducible, of dimension $\frac{1}{2}(l^2c_{ii}-l(l-1)) + \dim
\mathfrak{B}=\frac{1}{2}l^2c_{ii}=\frac{1}{2} \dim X_{l\alpha_i}$.
Since there is a natural finite surjective map $Y \rightarrow \N_{l
\alpha_i}$ whose restriction to $Y^{\circ}$ is injective,
$Y^{\circ}$ can be regarded as an open dense subset of $\N_{l
\alpha_i}$. Hence $\N_{l \alpha_i}$ is irreducible of dimension
$\frac{1}{2} \dim X_{l \alpha_i}$ as desired.
\end{proof}

\vspace{.1in}

Next, we introduce stratifications of $\N_{\alpha}$ (one for each $i
\in I$) as follows. Fix $i\in I$ and let $t$ be the number of loops
at $i$ in $\Omega$. Let $\R=\C\langle x_1, \ldots, x_t, y_1, \ldots,
y_t \rangle$ be the free associative algebra generated by $x_i, y_i$
$(i=1, \ldots, t)$. Write $\Omega_i^{loop}=\{\sigma_1, \ldots,
\sigma_t\}$.

For $B=(B_{h})_{h\in H} \in \N_{\alpha}$ and $f\in \R$, we define
\begin{equation*}
\begin{aligned}
f(B)& = f(B_{\sigma_1}, \ldots, B_{\sigma_t},
B_{\overline{\sigma_1}}, \ldots, B_{\overline{\sigma_t}}), \\
\C\langle B \rangle_{i}& = \set{ f(B)}{f \in \R }, \\
\varepsilon_i (B) &= \mathrm{codim}_{V_i} \Bigl(\C \langle B \rangle_i
\cdot \sum_{\substack{h\cl j\rightarrow i \\ j \neq i}} \Im
B_{h}\Bigr).
\end{aligned}
\end{equation*}
We put $\N_{\alpha,n,i}=\set{B \in \N_{\alpha}}{
\varepsilon_i(B)=n}$. It is clear that this defines a finite
stratification $\N_{\alpha}=\bigsqcup_{n \geq 0} \N_{\alpha,n,i}$
into locally closed subsets.

\vspace{.1in}

We choose an identification $V_i \overset{\sim} \longrightarrow
V_{i}^*$ and for $B\in X_\alpha$,
we define $B^*\in X_\alpha$ by $(B^*)_h = B_{\oh}^t\cl V_{\out(h)} \rightarrow
V_{\sink(h)}$ for $h\in H$. The map $B \mapsto B^*$ defines an involution
on $X_{\alpha}$ and $\N_{\alpha}$.
Set
\begin{equation*}
\varepsilon_i^*(B) = \varepsilon_i(B^*) = \mathrm{codim}_{V_i}\Bigl(
\C \langle B^t \rangle_i\cdot \sum_{\substack{h\cl i \rightarrow j \\ i
\neq j}} \Im B_{h}^t\Bigr) = \dim \bigcap_{\substack{h\cl i\rightarrow
j \\ i \neq j}} \mathrm{Ker} (B_{h} \cdot \C\langle B \rangle_i).
\end{equation*}
Here $\mathrm{Ker} (B_{h} \cdot \C\langle B \rangle_i)=
\set{v\in V_i}{B_{h} \cdot \C\langle B \rangle_iv=0}$.
Let $\Irr\, \N_{\alpha}$ denote the set of all irreducible
components of $\N_{\alpha}$. Of course, the $*$-involution induces
an involution on $\Irr\,\N_{\alpha}$ as well.
It does not depend on the choice of an isomorphism $V\simeq V^*$,
because $\N_\alpha$
is invariant by the action of $GL_\alpha$. 

For $\Lambda \in
\Irr\, \N_{\alpha}$, we define
\begin{equation*}
\varepsilon_i(\Lambda) = \varepsilon_i(B), \qquad
\varepsilon_i^*(\Lambda) = \varepsilon_i^*(B),
\end{equation*}
where $B$ is a generic point of $\Lambda$. By the definition,
$\Lambda \subset \N_{\alpha, \geq n, i}$ and $\Lambda \cap
\N_{\alpha, n,i}$ is dense in $\Lambda$ if and only if $n
=\varepsilon_i(\Lambda)$.

\begin{lem} \label{lem:varepsilon} The following statements hold.
\begin{itemize}
\item[(a)] $\varepsilon_i^*(\Lambda) = \varepsilon_i(\Lambda^*)$ for all
$i\in I$.

\item[(b)] If $\Lambda \in \Irr\,\N_{\alpha}$ and
$\varepsilon_i(\Lambda)=0$ for all $i\in I$, then $\alpha =0$ and
$\Lambda =0$.

\item[(c)] If $\Lambda \in \Irr\,\N_{\alpha}$ and
$\varepsilon_i^*(\Lambda)=0$ for all $i\in I$, then $\alpha =0$ and
$\Lambda =0$.

\end{itemize}
\end{lem}

\begin{proof} Statement (a) is obvious from the definitions. By (a),
statements (b) and (c) are equivalent. We now prove (b). Let
$\Lambda$ be as in (b), and assume that $\alpha \neq 0$. Let $B \in
\Lambda$ be a generic point, so that $\varepsilon_i(B)=0$ for all
$i$. By condition (i), there exists an $I$-graded complete flag
${F}=(F_0 \subset F_1 \cdots \subset F_{d})$ such that $B_h
(F_k) \subset F_k$ for all $h$ and $k$. In particular, $F_{d-1}$ is
stable under all operators $B_h$. Let $i_0 \in I$ be such that
$\underline{\dim} (F_d/F_{d-1})=\alpha_{i_0}$. We have
$\mathbf{C} \langle B \rangle_{i_0} \cdot \sum_{\substack{h\cl j \to i_0 \\
j \neq i_0}} \mathrm{Im}\, B_h \subset F_{d-1}$. But this yields
$\varepsilon_{i_0}(B) \geq 1$, which is a contradiction.
\end{proof}

\vskip 1cm


\section{Crystal Structure}

\vskip 3mm

For $i \in I$, $l \in \mathbf{N}$ and $\alpha = \sum d_i \alpha_i  \in
Q_{+}$, let
\begin{equation*}
\begin{aligned}
E_{\alpha; l\alpha_i}  = \big\{(B,B', B'', \phi', \phi) \;; \;
&B' \in \N_{\alpha}, B'' \in \N_{l \alpha_i}, B \in \N_{\alpha+ l\alpha_i}, \\
& 0 \longrightarrow V_{\alpha} \overset{\phi}\longrightarrow V_{\alpha+l\alpha_i}
\overset{\phi'} \longrightarrow V_{l \alpha_i} \longrightarrow
0 \ \ \text{is exact}, \\
& \phi \circ B' = B \circ \phi, \; \phi' \circ B=B'' \circ \phi'\big\}
\end{aligned}
\end{equation*}
be the space parametrizing the extensions of representations of the
quiver $(I,H)$. There are canonical maps
\begin{equation}\label{E:P1}
X_{\alpha} \times X_{l\alpha_i} \overset{p} \longleftarrow E_{\alpha; l\alpha_i}
\overset{q} \longrightarrow X_{\alpha+l\alpha_i}
\end{equation}
 given by
\begin{equation*}
p(B,B', B'', \phi, \phi') = (B', B''), \qquad
q(B,B', B'', \phi, \phi') = B
 \end{equation*}
and we may project further $p_{1}\cl E_{\alpha; l\alpha_i}
\longrightarrow X_{\alpha}$ . Put $\N_{\alpha; l\alpha_i}=q^{-1}(
\N_{\alpha+l\alpha_i})$. Then the diagram (\ref{E:P1}) restricts to
\begin{equation}\label{E:P2}
\begin{array}{c}
\xymatrix{
\N_{\alpha} \times \N_{l\alpha_i}\ar[dr]& \N_{\alpha; l\alpha_i}
\ar[l]_(.4){p}\ar[r]^q
\ar[d]^{p_1}
&\N_{\alpha+l\alpha_i}\\
&\N_{\alpha}}
\end{array}
\end{equation}
Observe that $p$ is not surjective in general. Indeed, if $i \in I$
is imaginary and $h \in \overline{\Omega}^{loop}$ is an edge loop at
$i$, then for any $(B,B',B'', \phi, \phi') \in \N_{\alpha;
l\alpha_i}$, the operator $B_{h}$ is regular semisimple, which
implies that the spectra of $B'_h$ and $B''_h$ are disjoint. Let us
denote by $\N_{\alpha} \times^{reg} \N_{l\alpha_i} \subset
\N_{\alpha} \times \N_{l \alpha_i}$ the open subset of pairs
$(B',B'')$ for which the operators $B'_h, B''_h$ have disjoint
spectra for any edge loop $h \in \overline{\Omega}^{loop}$ at $i$.

\vspace{.1in}

For any $n \geq 0$, we define $\N_{\alpha, n, i} \times^{reg} \N_{l
\alpha_i}$ to be the intersection of $\N_{\alpha} \times^{reg} \N_{l
\alpha_i}$ with $\N_{\alpha, n, i} \times \N_{l\alpha_i}$. The
locally closed subspace $\N_{\alpha, n,i; l \alpha_i}$ of
$\N_{\alpha; l\alpha_i}$ is defined in a similar fashion. That is,
$\N_{\alpha, n, i; l \alpha_i} = q^{-1}(\N_{\alpha + l \alpha_i, n,
i})$. Finally, let $Z_{\alpha, l\alpha_i}$ be the set of short exact
sequences (of $I$-graded vector spaces) $0 \longrightarrow
V_{\alpha} \overset{\phi}\longrightarrow V_{\alpha+l\alpha_i}
\overset{\phi'} \longrightarrow V_{l \alpha_i} \longrightarrow 0
$. Note that $GL_{\alpha + l \alpha_i}$ acts on $Z_{\alpha, l
\alpha_i}$ transitively.


\vspace{.1in}

\begin{prop}\label{P:dim} The following statements hold~{\rm :}

{\rm (a)} The restriction of $q$ to $\N_{\alpha,l,i; l\alpha_i}$ is
a $GL_{\alpha} \times GL_{l \alpha_i}$-principal bundle.

{\rm (b)} For $\alpha = \sum_{k} d_k \alpha_k$, the restriction of
$p$ to $\N_{\alpha,l,i; l\alpha_i}$ factors as
$$ \N_{\alpha,l,i;l\alpha_i} \stackrel{p'}{\longrightarrow}
(\N_{\alpha,0,i}\times^{reg} \N_{l\alpha_i}) \times Z_{\alpha,
l\alpha_i} \stackrel{p''}{\longrightarrow}
\N_{\alpha,0,i}\times^{reg} \N_{l\alpha_i},$$ where $p''$ is the
natural projection and $p'$ is an affine fibration of rank
$$r=l(\sum_{j \neq i} c_{ij}d_j + (c_{ii}-1)d_i)
= ( l\alpha_i, (\Id -A) \cdot \alpha ).$$
\end{prop}

\begin{proof} By the definition, if $(B,B',B'', \phi, \phi')$ belongs to
$\N_{\alpha, l, i; l\alpha_i}$, then there exists a unique
$B$-invariant submodule $W \subset V_{\alpha+l\alpha_i}$ such that
$\dim(V_{\alpha+l\alpha_i}/W)=l\alpha_i$. Namely, $W$ is the
submodule generated by $\bigoplus_{k \neq i} V_k$. This means that
$\mathrm{Im}(\phi)$ is uniquely determined, and thus $\phi, \phi'$ are
also determined up to a (free) $GL_{\alpha} \times
GL_{l\alpha_i}$-action, which proves (a).

We turn to (b). The map $p'$ is given by $(B,B',B'',\phi,\phi')
\mapsto ((B',B''),(\phi,\phi'))$. Note that by the above argument
the image of $p'$ indeed lies in $(\N_{\alpha,0,i} \times^{reg}
\N_{l\alpha_i})\times Z_{\alpha, l \alpha_i}$. Now let us fix
$(B',B'', \phi, \phi')$, set $W=\phi(V_{\alpha})$ and choose a
complement $U$ to $W$ in $(V_{\alpha+l\alpha_i})_i$. Thus
$\dim\,U=l$ and $\dim\, W_i = d_i$. We identify
$V_{l\alpha_i}$ with $U$ via $\phi'$, and $V_{\alpha}$ with
$\mathrm{Im}(\phi)$ via $\phi$. The fiber of $p'$ consists of
operators $B=(B_h)_h \in \N_{\alpha+l\alpha_i}$ which restrict to
$B'$ on $W$
and induces $B''$ on $U$. We may write
$$\begin{cases}
B_h=B'_h & \text{if}\; \out(h) \neq i,\\
B_h=B'_h + y_h & \text{if}\; \out(h)=i,\; \sink(h)=j \neq i,
\;\text{where}\; y_h\cl U \to (V_{\alpha})_j,\\
B_h=B'_h + B''_h + z_h & \text{if}\; \out(h)=\sink(h)=i, \;\text{where}\; z_h\cl U \to W.
\end{cases}$$

Given $(y_h, z_h)_h$ as above, the conditions for $B$ to belong to
$\N_{\alpha+l\alpha_i}$ are as follows~:
\begin{enumerate}
\item[i)] There exists a flag ${F}=(F_0 \subset F_1 \subset F_2 \cdots)$
such that $B_h(F_i) \subset F_i$ if $h \in \overline{\Omega}^{loop}$
and $B(F_i) \subset F_{i-1}$ otherwise,
\item[ii)] $\mu_k(B)=0$ for all $k$,
\item[iii)] $B_h$ is regular semisimple for $h \in \overline{\Omega}^{loop}$.
\end{enumerate}
The first condition is always satisfied: we may stack the flags
${F}',\, {F}''$ of $B'$ and $B''$ together; i.e., set
$F_n=F'_n$ for $n \leq \dim(V_\alpha)$ and
$F_{\dim(V_\alpha)+m}=V_{\alpha}\oplus F''_m$ for $m \leq l$.
The third condition is also always fulfilled because $(B',B'') \in
\N_{\alpha} \times^{reg} \N_{l\alpha_i}$. It remains to verify the
second condition, which reduces to $\mu_i(B)=0$. At this point we
distinguish two cases.

\vspace{.1in}

\noindent \underline{Case 1)} The vertex $i$ is real (i.e.,
$c_{ii}=0$).

Since $\mu_i(B')=0$, the moment map condition $\mu_i(B)=0$ reads
\begin{equation}\label{E:kop}
0 = \sum_{h\cl i \to j} \varepsilon(h)(B'_{\oh} y_h+
B'_{\oh}B'_h)=\sum_{h\cl j \to i}  \varepsilon(h)B'_{\oh} y_h.
\end{equation}
This implies $\mathrm{Im}( y\cl U \to \bigoplus_{h\cl i \to j}
(V_{\alpha})_j)$ lies in the kernel of the map
$$\sum_{h\cl i \to j} \varepsilon(h)B'_{\oh}\cl\bigoplus (V_{\alpha})_j \to W.$$
But since
$B' \in \N_{\alpha, i,0}$, we have
$$\dim\;\mathrm{Im}( \sum_{h\cl i
\to j} \varepsilon(h) B'_{\oh})=\dim(W_i)=d_i$$ and hence
$$\dim\;\mathrm{Ker}\big(\sum_{h\cl i \to j} \varepsilon(h) B'_{\oh}\big)
=\dim\big( \bigoplus_{h\cl i \to j} (V_{\alpha})_j\big) -d_i
=\sum_{j \neq i} c_{ij}d_j-d_i.$$ It follows that the fiber of $p'$
is an affine space of dimension $l (\sum_{j \neq i} c_{ij}d_j-d_i)$
as wanted.

\vspace{.1in}

\noindent \underline{Case 2)} The vertex $i$ is imaginary (i.e.,
$c_{ii}>0$).

Since $\mu_i(B')=\mu_i(B'')=0$, the moment map condition
$\mu_i(B)=0$ reads
\begin{equation}\label{E:kop2}
\begin{split}
0=&   \sum_{h \in \Omega_i^{loop}}
\big( [B'_{\oh},B'_h] + [B''_{\oh},B''_h] + (z_{\oh}B''_h -B'_hz_{\oh}) +(B'_{\oh}z_h -z_h B''_{\oh}) \big)\\
&\hspace{1.5in} +\sum_{\substack{h\cl i \to j\\ j\neq i}} \varepsilon(h) \big( B'_{\oh}y_h +B'_{\oh}B'_h\big) \\
=&\sum_{\substack{h\cl i \to j\\ j\neq i}} \varepsilon(h) \big(
B'_{\oh}y_h \big) +\sum_{h \in \Omega_i^{loop}} \left((z_{\oh}B''_h
-B'_hz_{\oh}) +(B'_{\oh}z_h -z_h B''_{\oh}) \right).
\end{split}
\end{equation}
Observe that, because $B''_{\oh} \in \mathrm{End}(U)$ and $B'_{\oh}\in
\mathrm{End}(W)$ have disjoint spectrum, the map
$$\mathrm{Hom}(U,W) \to \mathrm{Hom}(U,W), \quad z_h \mapsto B'_{\oh}z_h-z_h B''_{\oh}$$
is invertible for all $h \in \Omega_i^{loop}$. In particular, we may
choose $(y_h)_h, (z_{\oh})_h$ arbitrarily as well as all $z_{h}$
except for one, and uniquely solve (\ref{E:kop2}) for that last
$z_h$. Thus the space of solutions to (\ref{E:kop2}) is of dimension
$l (\sum_{j \neq i} c_{ij}d_j + (c_{ii}-1)d_i)$, which completes the
proof. \end{proof}

\vspace{.1in}

\begin{cor}\label{C:Lag} For any $\alpha \in Q_{+}$, the variety $\N_{\alpha}$ is Lagrangian.
\end{cor}
\begin{proof} We argue by induction on $\alpha$. The statement is true for $\alpha=l\alpha_i$
for some $i \in I$ and $l \in \mathbf{N}$ by Lemma~\ref{L:la}. Now
let $\alpha \in Q_{+}$ and let $\Lambda$ be an irreducible component
of $\N_{\alpha}$. By Lemma~\ref{lem:varepsilon}, there exists $i \in
I$ such that $\varepsilon_i(\Lambda) >0$. Set
$\varepsilon_i(\Lambda)=l$. Thus $\Lambda \cap \N_{\alpha, l,i}$ is
open and dense in $\Lambda$.
Put $\beta=\alpha-l\alpha_i$ and write $\beta=\sum_k
d_k \alpha_k$.

By Proposition~\ref{P:dim} (a), $q^{-1} (\Lambda \cap
\N_{\alpha,l,i})$ is an irreducible component of
$\N_{\beta,l,i;l\alpha_i}$ of dimension $\dim \Lambda +
\dim(GL_{\beta} \times GL_{l\alpha_i})$. Similarly, by
Proposition~\ref{P:dim} (b), $p$ is a smooth map with fibers
$Z_{\beta,l\alpha_i} \times \mathbb{A}^r$ with $r=( l\alpha_i,
(\Id-A) \cdot \beta )$, and thus $pq^{-1}(\Lambda \cap \N_{\alpha,l,i})$ is
an irreducible component of $\N_{\beta,0,i}
\times^{reg}\N_{l\alpha_i}$ of dimension
\begin{equation}\label{E:tryit}
\begin{split}
\dim\Lambda + \dim(GL_{\beta}\times GL_{l\alpha_i})-\dim
Z_{\beta,l\alpha_i}-r =\dim \Lambda + (l\alpha_i, (A-2\Id) \cdot
\beta).
\end{split}
\end{equation}
Recall that $\N_{\beta,0,i}$ is open in $\N_{\beta}$. Hence, by the
induction hypothesis, any irreducible component of $\N_{\beta,0,i}
\times^{reg} \N_{l\alpha_i}$ is of dimension
\begin{equation}\label{E:tryitagain}
\frac{1}{2} (\beta, (2\Id-A) \cdot \beta)  + \frac{1}{2} (l\alpha_i,
(2 \Id-A) \cdot l\alpha_i).
\end{equation}
Combining (\ref{E:tryit}) and (\ref{E:tryitagain}), we get the
dimension formula $\dim \Lambda= \frac{1}{2} (\alpha, (2\Id -A) \cdot
\alpha)$ as wanted.
\end{proof}

\vspace{.1in}

\begin{cor} For $\alpha \in Q_{+}$, $i \in I$ and $l \in \mathbf{N}$,
there is a one-to-one correspondence between the set of irreducible
components of $\N_{\alpha}$ satisfying $\varepsilon_i(\Lambda)=l$
and the set of irreducible components of
$\N_{\alpha-l\alpha_i,0,i}$.
\end{cor}

\begin{proof} By Proposition~\ref{P:dim}, the maps $p$ and $q$ in (\ref{E:P1}),
when restricted to $\N_{\alpha, l,i; l\alpha_i}$, are locally
trivial, smooth and with connected fibers. It follows that there
is a natural bijection between the sets of irreducible components
of $\N_{\alpha,l,i}$ and $\N_{\alpha-l\alpha_i,0,i} \times^{reg}
\N_{l\alpha_i}$. By Lemma~\ref{L:la}, $\N_{l\alpha_i}$ is
irreducible, and hence we obtain a bijection between the sets of
irreducible components of $\N_{\alpha,l,i}$ and
$\N_{\alpha-l\alpha_i,0,i}$. By Corollary~\ref{C:Lag}, any
irreducible component of $\N_{\alpha}$ is half dimensional and the
same is true for all irreducible components of $\N_{\alpha, l,i}$.
It follows that the irreducible components of $\N_{\alpha,l,i}$
are precisely the intersections of $\N_{\alpha,l,i}$ with the
irreducible components of $\N_{\alpha}$ satisfying
$\varepsilon_i(\Lambda)=l$, and we are done. \end{proof}

\vspace{.2in}

 Following \cite{KS97} and \cite{Lus91}, we will denote this
 one-to-one correspondence by $\Lambda \mapsto \eit^{l}(\Lambda)$,
 and define the Kashiwara operators $\eit,\fit$ on the set
 $\bigsqcup_{\alpha} \Irr\;\N_{\alpha} \cup \{0\}$ by
 \begin{equation}\label{eq:eit}
 \begin{aligned}
 \eit(\Lambda) &= \begin{cases} (\eit^{l-1})^{-1} \circ \eit^l (\Lambda)
 & \text{if}\; \varepsilon_i(\Lambda)=l>0,\\
 0 & \text{if}\; \varepsilon_i(\Lambda)=0,
 \end{cases}\\
 \fit(\Lambda)&= (\eit^{l+1})^{-1} \circ \eit^l(\Lambda) \qquad \text{if}\; \varepsilon(\Lambda)=l.
 \end{aligned}
 \end{equation}

Recall that we have fixed an identification $V_i \overset{\sim} \longrightarrow V_i^*$
and defined
$$B^* = (B_h^t)_{h\in H}, \qquad \varepsilon_i^*(B) =
\varepsilon_i(B^*).$$
We set
\begin{equation}\label{eq:eitstar}
\fit^* = * \circ \fit \circ *, \qquad \eit^* = * \circ \eit \circ *.
\end{equation}

\vspace{.1in}

The following Proposition is straightforward.

\begin{prop}\label{P:crys} \hfill

\begin{itemize}

\item[(a)] For any  $\Lambda \in \Irr \;\N_{\alpha}$, we have
$$\eit\fit (\Lambda) = \Lambda, \qquad
\varepsilon_i(\fit(\Lambda)) = \varepsilon_i(\Lambda) + 1,$$ and if
$\varepsilon_i(\Lambda)>0$, then
$$\fit\eit (\Lambda) = \Lambda, \qquad
\varepsilon_i(\eit(\Lambda)) = \varepsilon_i(\Lambda) - 1.$$

\item[(b)]  For any  $\Lambda \in \Irr \;\N_{\alpha}$, we have
$$\eit^*\fit^* (\Lambda) = \Lambda, \qquad
\varepsilon^*_i(\fit^*(\Lambda)) = \varepsilon^*_i(\Lambda) + 1,$$
and if $\varepsilon^*_i(\Lambda)>0$, then
$$\fit^*\eit^* (\Lambda) = \Lambda, \qquad
\varepsilon^*_i(\eit^*(\Lambda)) = \varepsilon^*_i(\Lambda) - 1.$$
\end{itemize}
\end{prop}
\qed

\vspace{.2in}

Let  $$\mathcal{B} = \coprod_{\alpha \in Q_{+}}
\mathcal{B}_{-\alpha} = \coprod_{\alpha \in Q_{+}} \Irr\;
\N_{\alpha}.$$ For $\Lambda \in \Irr\; \N_{\alpha}$, we define
\begin{equation*}
\begin{aligned}
& \wt(\Lambda)  = -\alpha, \\
& \varepsilon'_i (\Lambda) = \begin{cases} \varepsilon_i(\Lambda) \
\
& \text{if} \ \ i\in I^{re}, \\
0 \ \ & \text{if} \ \ i\in I^{im},
\end{cases}\\
& \varphi_i(\Lambda) = \langle h_i, \wt(\Lambda) \rangle +
\varepsilon'_i(\Lambda).
\end{aligned}
\end{equation*}
Then using Proposition \ref{P:crys}, we obtain:

\begin{thm} \label{thm:crystal}
The maps $\wt$, $\varepsilon'_i, \varphi_i$, $\fit, \eit$ $(i\in I)$
define a $U_q(\g)$-crystal structure on $\mathcal{B}$. \qed
\end{thm}

\vspace{.1in}

To finish this section, we introduce the following useful
notation. Let $\Lambda \in \Irr\;\N_{\alpha}$ and put
$l=\varepsilon_i(\Lambda)$. Let $B$ be a generic element of
$\Lambda$ so that $\varepsilon_i(B)=l$. We define the element
$B'=\eit^l(B)$ as follows. Let $W \subset V_i$ be the
characteristic subspace
$$W= \C \langle B \rangle_i \cdot \sum_{\substack{h\cl j \to i\\ j \neq i}} \mathrm{Im}\;B_h.$$
It is of dimension $d_i-l$, where $\alpha=\sum_k d_k \alpha_k$. Then
$B' \in \N_{\alpha-l\alpha_i}$ is the restriction of $B$ to the
subspace $V' \subset V$, where $V'_j=V_j$ for $j \neq i$ and
$V'_i=W$. Moreover, it is a generic element of $\eit^l(\Lambda)$. Of
course, a similar definition can be given for $\eit^*$ as well.

\vspace{.2in}


\vskip 1cm


\section{Geometric Construction of $B(\infty)$}

\vskip 3mm

Fix $i\in I$ and let $B_{i}=\set{b_i(-n)}{n\ge 0 }$ be the
elementary crystal. We define a map $\Psi_i \cl \mathcal{B}
\longrightarrow \mathcal{B} \ot B_{i}$ by
\begin{equation}
\Psi_i(\Lambda) = \eit^{*c} \Lambda \ot b_i(-c),
\end{equation}
where $c=\varepsilon_i^*(\Lambda)$.

\begin{thm} \label{thm:embedding}
The map $\Psi_i\cl \mathcal{B} \longrightarrow \mathcal{B} \ot B_i$ is
a strict crystal embedding.
\end{thm}

\begin{proof}
It is clear that the underlying map is injective. We will prove
\begin{equation} \label{eq:main}
\Psi_i(\ejt \Lambda) = \ejt \Psi_i(\Lambda) \ \ \text{for all} \ \
j\in I. \end{equation}

We distinguish several cases. If $i \in I^{re}$, then the proof of
\cite{KS97} goes through with no modification. So we assume that $i
\in I^{im}$.

\vspace{.1in}

\noindent \underline{Case 1)} $i \neq j$.

\vspace{.1in}

Since $\varepsilon_j'(b_i(-c)) = -\infty$, the tensor product rule
yields
$$\ejt \Psi_i(\Lambda) = \ejt \, (\eit^{*c} \Lambda \ot b_i(-c)) = \ejt
\eit^{*c} \Lambda \ot b_i(-c).$$

On the other hand, we have
$$\Psi_i(\ejt \Lambda)  = \eit^{*d} (\ejt \Lambda) \ot b_i(-d),$$
where $d=\varepsilon_i^*(\ejt \Lambda)$. Assume that
$\ejt(\Lambda)=0$; i.e., $\varepsilon_j(\Lambda)=0$. We claim
$\varepsilon_j(\eit^{*c}(\Lambda))=0$ as well so that $\ejt (
\eit^{*c}(\Lambda))=0$. Indeed, we have

\vspace{.1in}

\begin{lem} If $i \neq j$, then $\varepsilon_j(\eit^*(\Lambda))=\varepsilon_j(\Lambda)$
for every $\Lambda$.
\end{lem}
\begin{proof} It is enough to see that $\varepsilon_j(\eit^{*l} \Lambda)=\varepsilon_j(\Lambda)$
if $\varepsilon^*_i(\Lambda)=l$. Let $B$ be a generic point of
$\Lambda$ and let $B' =p_1q^{-1}(B)$ be the corresponding generic
point of $\eit^{*l}(\Lambda)$ (see \eqref{E:P2}). Then for any edge $h
\in H$ with $\sink(h) \neq i$, we have
$\mathrm{Im}(B_h)=\mathrm{Im}(B'_h)$. In particular,
$$\mathbf{C}\langle B \rangle_i \cdot \sum_{\substack{h\cl k \to j\\ k \neq j}} \mathrm{Im}\;B_h =
\mathbf{C}\langle B' \rangle_i \cdot \sum_{\substack{h\cl k \to j\\ k
\neq j}} \mathrm{Im}\;B'_h,$$ which gives the desired equality.
\end{proof}

Now assume that $\varepsilon_j(\Lambda)>0$. To prove
\eqref{eq:main}, thanks to the above lemma, we have only to show
that $\ejt \eit^{*c}(\Lambda)=\eit^{*c} \ejt(\Lambda)$, which is a
direct consequence of the following Lemma.

\vspace{.1in}

\begin{lem} If $i \neq j$, then $\ejt \eit^*(\Lambda)=\eit^*\ejt(\Lambda)$ for every $\Lambda$.
\end{lem}
\begin{proof} Put $a=\varepsilon_j(\Lambda)=\varepsilon_j(\eit^* (\Lambda))$.
It is enough to show that $\ejt^a
\eit^*(\Lambda)=\eit^*\ejt^a(\Lambda)$ for all $\Lambda$. Indeed, we
then have $\ejt^{a-1} \eit^* \ejt(\Lambda)=\eit^*\ejt^{a-1} \ejt
(\Lambda)=\eit^* \ejt^a(\Lambda)=\ejt^a
\eit^*(\Lambda)=\ejt^{a-1}\ejt \eit^*(\Lambda)$ from which we deduce
$\eit^* \ejt(\Lambda)=\ejt \eit^*(\Lambda)$. Similarly, it is enough
to prove that $\eit^{*b} \ejt^a (\Lambda)=\ejt^a\eit^{*b}(\Lambda)$,
where $b=\varepsilon_i^*(\Lambda)=\varepsilon_i^*(\ejt^a(\Lambda))$.
This can be done by chasing the diagram given in the following,
where $B \in \Lambda$ is a generic point, $\sigma = \alpha - a
\alpha_j$, $\beta = \alpha - b \alpha_i$, and the middle horizontal
line (resp.\ the middle vertical line) represents the short exact
sequence defining $\eit^{*b}(\Lambda)$ (resp.\ $\ejt^a(\Lambda)$).

$$\xymatrix{
&  & 0 \ar[d] & 0 \ar[d] &\\
0 \ar[r] & V_{b\alpha_i} \ar[r] \ar@{=}[d] & V_{\sigma} \ar [r]
\ar[d]&
\mathrm{Im}(V_{\sigma} \rightarrow V_{\beta}) \ar[r] \ar[d]  & 0\\
0 \ar[r] & V_{b \alpha_i} \ar[r]  & V_{\alpha} \ar[r] \ar[d]
& V_{\beta} \ar[r] \ar[d] & 0\\
& & V_{a \alpha_j}  \ar@{=}[r] \ar[d]
& V_{a \alpha_j}  \ar[d]  &\\
&  & 0 & 0 & }$$


\noindent Note that $V_{\beta}$ and $V_{\sigma}$ are uniquely
determined by $B$, and that the right column and the top
row represent $\ejt^a\eit^{*b}(\Lambda)$
and $\eit^{*b}\ejt^a(\Lambda)$, respectively.
\end{proof}

\vspace{.1in}

\noindent \underline{Case 2)} $i =j$.

Write
\begin{equation*}
\begin{aligned}
\wt(\Lambda) &  = -\alpha = -m \alpha_i -\sum_{k\neq i} m_k \alpha_k, \\
\wt(\eit^{*c}(\Lambda)) &= -\alpha + c \alpha_i = -(m-c) \alpha_i -
\sum_{k\neq i} m_k \alpha_k.
\end{aligned}
\end{equation*}
Thus
\begin{equation*}
\begin{aligned}
\varphi_i (\eit^{*c} (\Lambda)) & = (m-c) (-a_{ii}) + \sum_{k\neq i}
m_k (-a_{ik}), \\
\varepsilon'_i (b_i(-c))& = 0.
\end{aligned}
\end{equation*}

To prove our claim, we consider the following three cases:
\begin{itemize}
\item[(a)] $\varphi_i (\eit^{*c} (\Lambda)) \leq 0$,
\item[(b)] $ 0 < \varphi_i (\eit^{*c} (\Lambda)) \leq -a_{ii}$,
\item[(c)] $-a_{ii} < \varphi_i (\eit^{*c} (\Lambda))$.
\end{itemize}

\vspace{.2in}

(a) The condition (a) implies that $(m-c) (-a_{ii}) =0$ and $m_k
(-a_{ik})=0$ for all $k \neq i$. Hence $\dim V_{k}=0$ whenever there
is an arrow $h\cl i \rightarrow k$ $(k \neq i)$, which implies
$B_{h}=B_{h}^t =0$ for all such $h$. Hence
$$c = \varepsilon_i^*(\Lambda) = \mathrm{codim}_{V_i} 0 = \dim V_i =
m.$$ For similar dimension reasons, we have
$\varepsilon^*_i(\eit(\Lambda))=c-1$. Moreover, for any $l \geq
1$, our Lagrangian varieties decompose as
$$\N_{\alpha-l_{\alpha_i}} \simeq \N_{(m-l)\alpha_i} \times \N_{\alpha'},$$
where $\alpha'=\sum_{k \neq i} m_k \alpha_k$. The Kashiwara
operators $\eit, \eit^*$ act on the first component of this
decomposition. By Lemma~\ref{L:la}, we have
$\Irr\;\N_{l\alpha_i}=\{pt\}$ for all $l$, and it is clear
that $\eit^*=\eit \cl \Irr\;\N_{l\alpha_i} \stackrel{\sim}{\to}
\Irr\;\N_{(l-1)\alpha_i}$ for all $l$. In particular,
$\eit^{*(c-1)}\eit(\Lambda)=\eit^{*c}(\Lambda)$ and therefore
\begin{equation*}
\begin{aligned}
\eit \Psi_i(\Lambda) & = \eit (\eit^{*c} \Lambda \ot b_i(-c)) =
\eit^{*c} \Lambda \ot b_i(-c+1) \\
& = \eit^{*(c-1)}(\eit \Lambda) \ot b_i(-c+1) = \Psi_i(\eit
\Lambda),
\end{aligned}
\end{equation*}
which proves our claim in the case (a).

\vspace{.2in}

(b) In this case, we have $a_{ii}<0$.
We already know $m\ge c$. The condition (b) implies
$$(m-c-1) a_{ii} \ge \sum_{k\neq i} m_k (-a_{ik}) \ge 0.$$
Since $a_{ii}<0$, we must have $m \le c+1$; i.e., $m=c$ or $c+1$.

If $m=c+1$, we have $m_k = \dim V_k = 0$ whenever $a_{ik} \neq 0$
$(k\neq i)$, and hence
$$c=\varepsilon_i^*(B) = \mathrm{codim}_{V_i} \C \langle B^t \rangle_i
\sum_{h\cl i \rightarrow k, k \neq i} \Im B_{h}^t =\mathrm{codim}_{V_i}
0 = m = c+1,$$ which is a contradiction.

Hence $m=c$ and $\sum_{k\neq i} m_k (-a_{ik}) >0$. By the
definition, $m=c$ means that $B_h=0$ for all $h\cl i \to k, k \neq i$.
Moreover, there exists $k$ such that $m_k>0$ and $a_{ik} \neq 0$. We
claim that under these conditions $\varepsilon_i(\Lambda)=0$.
Indeed, since $B_h=0$ for all $h\cl i \to k$ on $\Lambda$, which is
coisotropic, all $B_{\overline{h}}$ for $\overline{h}\cl k \to i$ may
be chosen arbitrarily.
But because $B_{\sigma}$ is regular semisimple for any $\sigma \in
\overline{\Omega}_i^{loop}$, for a generic $B \in \Lambda$, we can
choose $B_{\oh}$ such that $\C \langle B \rangle_i \, \mathrm{Im}
B_{\oh} = V_{i}$. It follows that
$$\mathbf{C}\langle B \rangle_i \cdot \sum_{\substack{h\cl k \to i\\ k \neq i}}
\mathrm{Im}\;B_h=V_i$$
as wanted. Hence $\eit(\Lambda)=0$ and by the tensor product rule,
we obtain
$$\eit \Psi_i(\Lambda) = \eit (\eit^{*c} \Lambda \ot b_i(-c)) = 0
= \Psi_i (\eit \Lambda).$$

\vspace{.2in}

(c) If $m=c$, the condition (c) implies there exists $k \neq i$
such that $m_k > 0$, $a_{ik} \neq 0$. By the same argument in (b),
one can deduce $\eit \Lambda =0$. On the other hand, since $m=c$,
we have $\dim_i(\eit^{*c}(\Lambda))=0$ and therefore
$\varepsilon_i(\eit^{*c} \Lambda) = 0$, which implies $\eit
(\eit^{*c} \Lambda) =0$. Hence by the tensor product rule, we have
\begin{equation*}
\begin{aligned}
\eit \Psi_i(\Lambda) &= \eit(\eit^{*c} \Lambda \ot b_i(-c) = \eit
(\eit^{*c} \Lambda) \ot b_i(-c) \\
&= 0 \ot b_i(-c) = 0 = \Psi_i(\eit \Lambda).
\end{aligned}
\end{equation*}

Let us now assume that $m>c$. By the tensor product rule again, we
have to prove that $\eit^{*c}\eit(\Lambda)=\eit \eit^{*c}(\Lambda)$
and $\varepsilon_i^{*}(\eit(\Lambda)) =c$. Let $B$ be a generic
element of $\Lambda$ and let us consider the characteristic spaces
$$W=\mathbf{C}\langle B \rangle_i \cdot \sum_{\substack{h\cl j \to i\\ j \neq i}} \mathrm{Im}\; B_h,$$
$$U=\bigcap_{\substack{h\cl i \to j\\ j \neq i}} \mathrm{Ker}
\left(B_h \cdot \mathbf{C}\langle B \rangle_i \right).$$ Then we have
$\dim\;U=\varepsilon_i^*(B)=c$ and
$\mathrm{codim}_{V_i}\;W=\varepsilon_i(B)$.








We first claim $U \subset W$. Let $d=\varepsilon_i(\Lambda)$ and put
$B'=\eit^d(B)$. The operator $B'$ acts on the subspace $V' \subset
V$ with $V'_k=V_k$ for $k \neq i$ and $V'_i=W$. Moreover
$B\vert_{V'}=B'$ and $B$ can be viewed as a (generic) element in the
fiber of $\{B'\} \times \N_{d \alpha_i}$ under the map $B \mapsto
(B|_{V'}, B|_{V/V'})$. Take $\sigma \in \overline{\Omega}_i^{loop}$
so that $B_{\sigma}$ is regular semisimple. Since $B_{\sigma}$
preserves $W$, we may choose a splitting $V_i=W \oplus T$ invariant
under $B_{\sigma}$. Let $\{v_1, \ldots, v_t\}$ be the basis of $T$
consisting of $B_{\sigma}$-eigenvectors. Since we have assumed
$m>c$, $U \neq V_i$, and there exists $k \neq i$ such that $a_{ik}
\neq 0$ and $V_k \neq \{0\}$. Let $h \cl i \to k$ be an edge in $H$.
Since $B_{h}|_{T}\cl T \to V_k$ may be chosen arbitrarily (see the
proof of Proposition~\ref{P:dim}), $B_h(v_l) \neq 0$ $(l=1, \ldots,
t)$ for generic $B$. Hence $U$ does not contain $v_1, \ldots, v_t$.
Since $U$ is invariant under $B_{\sigma}$, $U \subset W$.



By the above claim, we have $d\seteq\varepsilon_i(\Lambda)
=\varepsilon_i(\eit^{*c}(\Lambda))$ and for a generic $B \in
\Lambda$, $B'\seteq \eit^d \eit^{*c}(B)$ is the operator induced by $B$
on the space $\bigoplus_{j \neq i} V_j \oplus W/U$. On the the hand,
we have $\varepsilon_i^*(\eit(\Lambda))=\dim\;U=c$ and
$\varepsilon_i(\eit^{*c} \eit (\Lambda))=\varepsilon_i(\Lambda)-1$.
It is easy to see that $B''\seteq\eit^{d-1} \eit^{*c}\eit (B)$
coincides with $B'$ and hence that $\eit^{*c}\eit(\Lambda) = \eit
\eit^{*c}(\Lambda)$. We are done with case (c), and hence with the
proof of Theorem~\ref{thm:embedding}
\end{proof}

\vskip 2mm

Now we obtain the main result of this paper.

\begin{thm}\label{thm:main}
There exists a crystal isomorphism
$$\mathcal{B} = \coprod_{\alpha \in Q_{+}} \Irr \, \N_{\alpha}
\overset {\sim} \longrightarrow B(\infty).$$
\end{thm}

\begin{proof}
Our assertion follows from Theorem \ref{thm:B(infty)}, Lemma
\ref{lem:varepsilon}, Theorem \ref{thm:crystal}, and Theorem
\ref{thm:embedding}.
\end{proof}

\vskip 3mm

\begin{example} \hfill

(a) Let $(I, H)$ be the quiver with one vertex $I=\{i\}$ and $2t$
edge loops in $H$. Thus there are $t$ edge loops in $\Omega$ and
$\overline{\Omega}$, respectively, and the corresponding
Borcherds-Cartan matrix is $A=(2 - 2t)$.

For $l \ge 1$, let $V=\C^{l}$ be the $I$-graded vector space with
$\underline{\dim} V = l \alpha_i$. By Lemma~\ref{L:la}, the variety
$\N_{l \alpha_i}$ is irreducible and Lagrangian. Moreover, if
$B=(B_i, \overline{B_i})_{1 \le i \le t}$ is a generic element of
$\N_{l \alpha_i}$, then we have
\begin{equation*}
\begin{aligned}
& \varepsilon_i(B)= \mathrm{codim}_{V_i} 0 = l, \\
& \varepsilon_i^*(B)= \varepsilon_i(B^{t}) = \mathrm{codim}_{V_i} 0 =
l.
\end{aligned}
\end{equation*}
Hence we have
$$\Psi_i(\fit^l \textbf{1}) = \textbf{1} \otimes b_i(-l),$$
and the crystal structure on $B(\infty)$ is given as follows.

$$\textbf{1} \overset{i} \longrightarrow \fit \textbf{1} =
\N_{\alpha_i} \overset{i} \longrightarrow \cdots\cdots \overset{i}
\longrightarrow \fit^l \textbf{1} = \N_{l \alpha_i} \overset{i}
\longrightarrow \cdots\cdots
$$

\vskip 3mm

(b) Let $(I, H)$ be a quiver, where $I=\{i, j\}$ and $H$ consists of
two edge loops at $i$, one arrow from $i$ to $j$, and one arrow from
$j$ to $i$. Thus the corresponding Borcherds-Cartan matrix is
$A=\left(\begin{matrix} 0 & -1 \\ -1 & 2 \end{matrix} \right)$. We
choose an orientation $\Omega$ consisting of an edge loop at $i$ and
an arrow from $i$ to $j$.

We will compute $\fjt \fit^2 \textbf{1}$. Let $V=V_i \oplus V_j =
\C^2 \oplus \C$ be the $I$-graded vector space with
$\underline{\dim} V = 2 \alpha_i + \alpha_j$, and let $B=(B_1,
\overline{B_1}, B_2, \overline{B_2})$ be a generic element in $\N_{2
\alpha_i + \alpha_j}$, where $B_1, \overline{B_1}\cl \C^2 \rightarrow
\C^2$, $B_2\cl \C^2 \rightarrow \C$, $\overline{B_2}\cl \C \rightarrow
\C^2$ are the linear maps satisfying the conditions for
$\N_{2\alpha_i + \alpha_j}$. By applying the Kashiwara operators
successively to $\textbf{1}$, one can deduce that, for an
appropriate basis of $V$, $B$ has the form 
$$B=(B_1, \overline{B_1}, B_2, \overline{B_2}) = \left(0,
\left(\begin{matrix} \la & a \\ 0 & \mu
\end{matrix} \right), 0, \left( \begin{matrix} x \\ y
\end{matrix} \right) \right),$$
where $a, \la, \mu, x, y \in \C$, $\la \neq \mu$. Note that
\begin{equation*}
\begin{aligned}
& \varepsilon_i(B)= \mathrm{codim}_{V_i}\C \langle B_1, \overline{B_1} \rangle_i \mathrm{Im} \overline{B_2} =0, \\
& \varepsilon_i^*(B)= \varepsilon_i(B^{t}) = \mathrm{codim}_{V_i} \C
\langle B_1^{t}, \overline{B_1}^{t} \rangle 0  = 2, \\
& \varepsilon_j(B) = \mathrm{codim}_{V_j} \mathrm{Im}B_2 =
\mathrm{codim}_{V_j} 0 = 1, \\
& \varepsilon_j^{*} (B) = \mathrm{codim}_{V_j}
\mathrm{Im}\overline{B_2}^t = 0.
\end{aligned}
\end{equation*}
Hence we have
\begin{equation*}
\begin{aligned}
& \Psi_i(\fjt \fit^2 \textbf{1}) = \eit^{*2} (\fjt \fit^2
\textbf{1})
\ot b_i(-2) = \fjt \textbf{1} \ot b_i(-2), \\
& \Psi_j(\fjt \fit^2 \textbf{1}) = \fjt \fit^2 \textbf{1} \ot
b_j(0).
\end{aligned}
\end{equation*}

\qed
\end{example}

\vskip 1cm

\providecommand{\bysame}{\leavevmode\hbox
to3em{\hrulefill}\thinspace}

\end{document}
